\documentclass[10pt]{article}
\usepackage{fullpage,url,amsmath,amsfonts,amssymb, tedmath1}
\usepackage{amsmath,amsfonts,amssymb, tedmath1}
\newcommand{\cc}{\mathbb{C}}

\newcommand{\cir}{\,\text{circ}}

\newcommand{\lcm}{\text{lcm}}

\title{Non-separable matrix builders for signal processing, quantum information and MIMO applications
}
\author{ Ted
 Hurley\footnote{National Universiy of Ireland Galway, email:
 Ted.Hurley@NuiGalway.ie }  \, \&
Barry Hurley\footnote{Friar's Hill, Galway,  email: Barryj\_2000@yahoo.co.uk}
}
\date{}

\newcommand{\s}{\sigma}
\begin{document}

\maketitle

 


\begin{abstract} Matrices are built and designed  by applying
  procedures from lower order matrices. Matrix tensor products, direct sums or multiplication of matrices are such procedures and a matrix
  built from these is said  to be a {\em
    separable} matrix. A {\em non-separable} matrix is a matrix  which is not separable and is often referred to  as  {\em an entangled matrix}. The matrices built may retain properties of the lower order matrices or  may also acquire new desired properties not inherent in the constituents.   
  Here design methods for non-separable  matrices of required types are derived. These can
  retain properties of lower order matrices  or  have new desirable
  properties.  Infinite series of required type non-separable matrices are  constructible by the general methods. 
  
  Non-separable  matrices of required types are required for applications and other uses; they can capture the structure in a unique way and thus perform much better than separable matrices.   
  General new methods are developed with which  to construct {\em multidimensional entangled paraunitary
  matrices}; these have applications for wavelet and filter
  bank design. The  constructions are used  to design new systems of non-separable unitary matrices; these have applications in quantum information theory. Some consequences include the design of full diversity constellations of  unitary matrices, which are used in MIMO systems, and methods to design infinite series of special types of Hadamard matrices.

\end{abstract}

\section{Introduction}

Matrices are built and designed from lower order matrices.
A {\em separable matrix} is a matrix built using a direct sum,
tensor product or multiplication.
These procedures preserve many properties of the constituents; 
for example the tensor product or direct sum  of invertible matrices
is invertible  and the tensor product of unitary
matrices is  a unitary matrices. A design procedure may or may not retain properties of the constituents or 
may acquire new properties not inherent in the constituents.  
A
non-separable matrix is often referred to as an {\em entangled matrix}.  
Non-separable matrices with specific properties are required for structures and applications. These are required for specific application purposes in for example 
{\em quantum information theory} 
or {\em signal processing}.  


Building blocks for non-separable matrices to specific requirements and types are presented. These may retain specific properties
of the constituents or, more importantly, may acquire new desired properties not inherent in the constituents. 
The   constructions enable infinite series of entangled matrices of a required type to be
built. 

Building blocks for {\em paraunitary matrices} are  fundamental in {\em signal
  processing}. The concept of a paraunitary matrix is fundamental in
  this area and non-separable  paraunitary matrices are required for {\em better performance}. 
Filter banks play an important role in signal processing but
multidimensional entangled  filter banks have been hard to design.   
In   the huge research area of multirate filterbanks and  wavelets,
paraunitary matrices play a fundamental role, see for example \cite{strang,zhou}. See section \ref{para} below for further background.    

Non-separable unitary matrices have applications in diverse areas such
as for example in {\em quantum information theory},  \cite{quan}. 

Results and constructions on special unitary matrices are carried over
to give new perspectives on the design of Hadamard special
matrices, such as skew Hadamard and symmetric Hadamard matrices, section
\ref{hada}. 

Full diversity sets of constellations of unitary matrices of many
forms and of  good  quality are designed from the constructions,  see
section \ref{spacetime}. These are
required for MIMO (multiple input, multiple output) systems. The ones
  designed here have excellent {\em quality} (a defined concept)   and
  infinite series of 
such   constellations may be designed.    


Non-separable multidimensional systems are designed 
and these  can capture geometric structure rather than those constructed from one dimensional schemes using separable constructs.

Infinite series of such required matrices  can be built from the constructions  and processes  may  be selected.  
Applications to  cryptography are inherent but are not dealt with here. 




\quad
The two basic constructions are
described separately in sections \ref{gyt} and \ref{dita}; these   
 are combined to give additional 
design techniques in  section \ref{joint}. 
Applications  
are expanded on substantially in later sections. The
non-separability/entanglement  concept is important at all stages.

 
A formula for  the determinant of a designed  square matrix is
obtained  in subsection \ref{det} and is  of independent interest;
it may be considered as a major generalisation of the determinant formula 
for a tensor product. This determinantal formula  has a number of applications including the computation of the quality  of full diversity sets of constellations as constructed in Section \ref{spacetime}; these are used in MIMO (multiple input, multiple output) schemes.

\subsection{Notation} Basic algebra notation and background may be
found in many books on matrix theory or linear algebra but  also found extensively online nowadays. Matrices are formed over rings in general including over polynomial rings in particular here. 

$A\T$ denotes the transpose of the matrix $A$. For a matrix $A$ over $\C$, $A^*$ denotes the complex conjugate transposed of $A$; over other rings by convention $A^*=A\T$. Now $I_n$ denotes the identity $n\ti n$ matrix which is also denoted by $I$ when the size is understood. The notation $1_R$ is used for the identity of the ring $R$ which is abbreviated to $1$ when the ring is understood. 
Say   $A$ is a {\em unitary $n\ti n$ matrix} provided $AA^*=I_n$, and say $H$ is a {\em symmetric} (often called  Hermitian) matrix provided $H^*=H$.

A one-dimensional (1D) paraunitary matrix over $\C$ is a square
matrix $U(z)$ satisfying 
$U(z){U}^*(z^{-1}) = 1$.  
In general a $k$-dimensional (kD) paraunitary
matrix over $\cc$ is a matrix $U(\bf{z})$, where ${\bf{z}}=(z_1,z_2,
\ldots, z_k)$ is 
a vector of (commuting) variables $\{z_1,z_2, \ldots, z_k\}$,  such that
$U({\bf{z}}){U}^*({\bf{z}^{-1}}) = 1$ with  the definition 
 ${\bf{z}^{-1}} = (z_1^{-1},z_2^{-1}, \ldots, z_k^{-1})$. 
Over fields other than $\C$ a
 paraunitary matrix  is a matrix $U(\bf{z})$ satisfying 
 $U({\bf{z}}){U}\T({\bf{z}^{-1}}) = 1$.   

 An idempotent matrix $E$ is a matrix satisfying $E^2 = E$. The
 idempotent is {\em symmetric} provided $E^*=E$; idempotents considered  here are symmetric. A complete orthogonal symmetric idempotent (COSI) set
 is a set of $n\ti n$  matrices $\{E_1, E_2, \ldots, E_k\}$ where each $E_i$ is a symmetric idempotent, $E_iE_j=0=E_jE_i$ for $i\neq j$ and $E_1+E_2+\ldots + E_k=I_n$.
 Further definitions are given as required  within sections. Definitions related to constellations of unitary matrices are given in Section \ref{spacetime};  definitions related to special  Hadamard matrices, (real or complex) are given in section \ref{hada}. 


 \section{The constructions}\label{prod}

 The basic designs use (i) COSI sets, section \ref{gyt},  and (ii) methods
 related to Di\c{t}\u{a} type construction, section \ref{dita}. These are then combined. 

\subsection{Design with COSI}\label{gyt}
Methods are now developed with which to design and construct required types of non-separable matrices  using  complete orthogonal symmetric idempotent (COSI)
sets. 
Using COSI sets for constructing series of unitary and paraunitary matrices was initiated in \cite{tedbarry}.   

 
  \begin{proposition}\label{idem} Let $\{E_1,E_2, \ldots, E_k\}$ be a COSI set in $\C_n$. 
Define $G= \left(\begin{smallmatrix} E_{{11}} & E_{{12}} & \ldots & E_{{1k}} \\ E_{{21}} & E_{{22}} & \ldots & E_{{2k}} \\ \vdots & \vdots & \vdots & \vdots \\ E_{{k1}} & E_{{k2}} & \ldots & E_{{kk}} \end{smallmatrix}\right)$  where 
each $E_j$ appears once in each (block) row and once in each (block) column. 
 Then $G$ is a unitary $nk \times nk$ matrix.
  \end{proposition}
\begin{proof} Take the block inner product of two different rows of
 blocks. The $E_i$ are orthogonal to one another so the result is
 $0$. Take the block inner product of the row, $j$, of blocks  with
 itself. This gives 
$E_{j1}^2+E_{j2}^2+\ldots + E_{jk}^2 = E_{j1}+E_{j2}+\ldots + E_{jk} = 1 (=I_n)$. Hence  $GG^* = 1 ( =I_{nk})$.
\end{proof}


  A block circulant matrix is one of the form $\left(\begin{smallmatrix} A_1 & A_2 & \ldots &  A_{n} \\ A_n & A_1 & \ldots & A_{n-1} \\ \vdots & \vdots & \vdots & \vdots \\ A_2 & A_3 & \ldots & A_1 \end{smallmatrix}\right)$ where the $A_i$ are blocks of the same size. A reverse circulant block matrix is one of the form $\left(\begin{smallmatrix} A_1 & A_2 & \ldots & A_{n-1} & A_{n} \\ A_2 & A_1 & \ldots & A_{n} & A_1\\ \vdots & \vdots & \vdots & \vdots & \vdots \\ A_n & A_1 & \ldots & A_{n-2} & A_{n-1} \end{smallmatrix}\right)$ where the $A_i$ are blocks of the same size. A circulant block matrix may be transformed into a reverse circulant block matrix by block row operations. 
\\ For example
 $\left(\begin{smallmatrix} E_1& E_2&E_3 & E_4 \\ E_4& E_1&E_2 &E_3\\ E_3&E_4&E_1& E_2 \\ E_2&E_3&E_4&E_1\end{smallmatrix}\right) \longleftrightarrow \left(\begin{smallmatrix} E_1 & E_2 & E_ 3 & E_4 \\ E_2&E_3&E_4&E_1\\ E_3 & E_4 &E_1& E_2 \\ E_4&E_1&E_2&E_3 \end{smallmatrix}\right)$, where $\longleftrightarrow$ here indicates that one can be  obtained from the other by block row operations. The one on the left is block circulant and the one on the right is reverse block circulant. 

      In particular given a COSI set  $\{E_1, E_2, \ldots, E_k\}$, block circulant unitary matrices and block reverse circulant unitary matrices may be formed. Note that the block reverse circulant matrix is {\em symmetric} as the $E_i$ are  symmetric. 

 When  variables are  attached to the $E_i$ a paraunitary matrix is
 obtained; when elements of modulus $1$ are attached to the $E_i$ a
 unitary matrix is obtained.

 For a variable $\al$ define $\al^*=\al^{-1}$.

   \begin{proposition}\label{idem1}  Let $\{E_1,E_2, \ldots, E_k\}$ be
     a COSI set in $\C_n$. \\ Define $G= \left(\begin{smallmatrix} E_{{11}}\al_{11} & E_{{12}}\al_{12} & \ldots & E_{{1k}}\al_{1k} \\ E_{{21}}\al_{21} & E_{{22}}\al_{22} & \ldots & E_{{2k}}\al_{2k} \\ \vdots & \vdots & \vdots & \vdots \\ E_{{k1}}\al_{k1} & E_{{k2}}\al_{k2} & \ldots & E_{{kk}}\al_{kk} \end{smallmatrix}\right)$ 
   where  each $E_k$ appears once in each (block) row and once in each (block) column.

     (i) Let the $\al_{ij}$ be  variables. Then $GG^* = I$ so that $G$ is a paraunitary matrix. 

     (ii) If $|\al_{ij}| =1$ for each $\al_{ij}$, then $G$ is a unitary matrix.
   \end{proposition}
   
   The proof is similar to the proof of Proposition \ref{idem}.
   
   Block circulant and block reverse circulant matrices may be formed. The reverse circulant block matrix is symmetric provided $\al_{ij} = \al_{ji}^*$.  

    There is no limit on size and large constructions may also
    be formulated iteratively. The designs are direct and efficient. 

   \begin{example}\label{one}
Let $E_0=\frac{1}{2}\left(\begin{smallmatrix} 1& 1 \\ 1& 1
			  \end{smallmatrix}\right), E_1 =
\frac{1}{2}\left(\begin{smallmatrix} 1& -1 \\ -1& 1 \end{smallmatrix}\right)$.
Then $\{E_1,E_2\}$ is a COSI set.
Define  
 $W = \left(\begin{smallmatrix} xE_0 & yE_1 \\ zE_1 &
		 tE_0\end{smallmatrix}\right)=
		 \frac{1}{2}\left(\begin{smallmatrix}x&x & y& -y \\ x&x &
				     -y & y \\ z&-z&t&t \\
				  -z&z&t&t\end{smallmatrix}\right)$. 
Then $WW^*= I_4$. 


   Let 
	$x=1=t=y=z$  in  $W$ and the following matrix is obtained: 
$H= \frac{1}{2}\left(\begin{smallmatrix} 1 & 1 & 1 & -1 \\ 1 & 1 & -1& 1 \\ 1&-1
	     & 1 & 1 \\ -1 & 1 & 1 & 1 \end{smallmatrix}\right)$; this is a
	common matrix used, or given as an example, in quantum theory
        as a non-separable/entangled matrix. 

  \end{example}
\begin{example}\label{two}
Let    $Q_0=\frac{1}{2}\left(\begin{smallmatrix} 1& i \\ -i& 1
			  \end{smallmatrix}\right), Q_1 =
\frac{1}{2}\left(\begin{smallmatrix} 1& -i \\ i& 1 \end{smallmatrix}\right)$.
Then $\{Q_0,Q_1\}$ is a COSI set.
Define $Q = \left(\begin{smallmatrix} xQ_0 & yQ_1 \\ zQ_1 &
		 tQ_0\end{smallmatrix}\right)$.          
Then $Q$ is a paraunitary matrix. Now letting the variables have complex
 values of modulus $1$  gives rise to 
complex Hadamard matrices as for example 
$\left(\begin{smallmatrix} 1 & i & 1 & -i \\ -i & 1 & i& 1 \\ 1&-i
	     & 1 & i \\ i & 1 & -i & 1 \end{smallmatrix}\right)$.  
\end{example}





\begin{example} Consider the matrices in Example \ref{one} where  $E_1=\frac{1}{2}\left(\begin{smallmatrix} 1& 1 \\ 1& 1 \end{smallmatrix}\right), E_2 =
\frac{1}{2}\left(\begin{smallmatrix} 1& -1 \\ -1& 1 \end{smallmatrix}\right)$.
Then $G=\left(\begin{smallmatrix} E_1 & E_2 \\ E_2 & E_1 \end{smallmatrix}\right)$ and $H=\left(\begin{smallmatrix} E_2 & E_1 \\ E_1& E_2 \end{smallmatrix}\right)$ are unitary matrices. Then $2G, 2H$ are Hadamard real $4\ti 4$ matrices.
 Form $F_i=u_iu_i^*$ where  $\{u_1,u_2,u_3,u_4\}$ are the columns of
 $G$ and then $\{F_1,F_2,F_3,F_4\}$ is a COSI set. These may then be
 used to form $16\ti 16$ unitary matrices; the entries are $\pm
 \frac{1}{4}$ and thus $4$ times these matrices are Hadamard $16\ti
 16$ matrices. In particular $G= \left(\begin{smallmatrix}F_1 & F_2 & F_3 & F_4
   \\ F_2 & F_3 & F_4& F_1 \\ F_3 & F_4 & F_1 & F_2 \\ F_4 & F_1 &
   F_2& F_3 \end{smallmatrix}\right)$ is a symmetric unitary matrix and thus $4 G$ is a symmetric Hadamard real matrix.
\end{example}

  Here is a list of some properties of idempotents which are well-known or easily deduced.  
\begin{itemize}
  \item 
    Let $\{u_1, u_2, \ldots, u_k\}$ be an orthonormal set of column
    vectors. Define $E_i=u_iu_i^*$ and then $\{E_1, E_2, \ldots,
    E_k\}$ is an orthogonal symmetric set of idempotents. If $S=\{E_1, E_2, \ldots, E_k\}$ is not complete, set $E= (I - E_1 - E_2 - \ldots - E_k)$ and then $\{E_1, E_2, \ldots, E_k, E\}$ is a COSI set.
  \item If $\{E_1, E_2, \ldots, E_k\}$ is an orthogonal symmetric idempotent set, then \\ $\rank(\sum_{i=1}^k E_i) = \sum_{i=1}^k(\rank E_i)$.
    \item If $E$ is an idempotent of $\rank k$ then $E$ is the sum of $k$ orthogonal idempotents of $\rank 1$. A method for writing such an idempotent as the sum of $\rank 1$ idempotents is given in \cite{ted1}.\footnote{The summation is not unique but a unique expression can be obtained by expressing  the idempotent as the sum of $\rank 1$ idempotents with increasing initial zeros.}   
\item 
      When $U$ is unitary, its columns $\{u_1, u_2, \ldots, u_n\}$
      form an orthonormal basis and thus $\{E_1, E_2, \ldots, E_n\}$
      with $E_i=u_iu_i^*$ is a COSI set which may then be used to form unitary or paraunitary matrices.
    \item If $\{E,F\}$ are orthogonal idempotents then $E+F$ is an idempotent orthogonal to any idempotent which is orthogonal to both $E,F$. Thus if $\{E,F, K_1, K_2, \ldots, K_t\}$ is an orthogonal idempotent set so is $\{E+F, K_1,K_2,\ldots, K_t\}$ and if $\{E,F, K_1,K_2 \ldots, K_t\}$ is a COSI set so is $\{E+F, K_1, \ldots, K_t\}$   
      \end{itemize}
      
Orthogonal idempotents may be combined to form new idempotents and
thus 
elements in a COSI set may be combined to form a new COSI set with a
smaller number of elements but of the same size. This
new COSI set may then be used to design unitary, paraunitary matrices
and others. The following examples illustrate the general method.

Denote the circulant matrix $\left(\begin{smallmatrix} a_1 & a_2 &
  \ldots & a_k \\ a_k & a_1& \ldots & a_2 \\ \vdots & \vdots & \vdots
  & \vdots \\ a_2 & a_3 & \ldots & a_1 \end{smallmatrix}\right)$ by
$\cir (a_1,a_2,\ldots, a_k)$.

Note that if $\om = e^{i\theta}$ then $\om+\om^* = 2\cos \theta$. 
\begin{example} Denote the columns of the $5\ti 5$ normalised Fourier
  matrix by $\{u_0,u_1,u_2,u_3,u_4\}$. Define $E_i=u_iu_i^*$. Then
  $E_i=\frac{1}{5}\cir(1, \om^{4i},\om^{3i},\om^{2i},\om^{i})$ where
  $\om=e^{i \frac{2\pi}{5}}$ is a primitive $5^{th}$ root of $1$ and $\{E_0,E_1,E_2,E_3,E_4\}$ is a COSI set. Now combine $\{E_1, E_4\}$ and $\{E_2,E_3\}$ to get the COSI set $S=\{E_0, E_1',E_2'\}$ where $E_1'= E_1 + E_4, E_2'=E_2+E_3$. The elements in this COSI set are circulant matrices also but in addition have real entries:
  $E_1' = \frac{2}{5}\cir (1, \cos \theta, \cos 2\theta, \cos 3\theta,\cos 4\theta), E_2' = \frac{2}{5}\cir (1, \cos 2\theta, \cos 4 \theta, \cos \theta,\cos 3\theta)$ where $\theta = \frac{2\pi}{5}$. It is noted  that $\cos 4\theta = \cos \theta, \cos 3\theta = \cos 2\theta$  -- which could be deduced from the fact that  $\{E_1',E_2'\}$ are symmetric! 

  This $S$ can then be  used to design unitary and paraunitary matrices with real coefficients as for example $\left(\begin{smallmatrix} E_0 & E_1' & E_2' \\ E_1' & E_2' & E_0 \\ E_2'  & E_0&E_1' \end{smallmatrix}\right)$. 
\end{example}
\begin{example}  Let  $\{u_0,u_1, \ldots, u_5\}$ be the columns  of the normalised Fourier $6\ti 6$ matrix and form $E_i=u_iu_i^*$. Combine $\{E_1, E_5\}$ and $\{E_2,E_4\}$ to obtain the real COSI set $S=\{E_0,E_1',E_3, E_2'\}$ where $E_1'=E_1+E_5, E_2'=E_2+E_4$. 
Now a primitive $6^{th}$ root of $1$ is $\om=e^{i\frac{2\pi}{6}}= \cos \frac{2\pi}{6} + i\sin \frac{2\pi}{6}$ and $\cos \frac{2\pi}{6} = \frac{1}{2}, \cos \frac{4\pi}{6} = -\frac{1}{2}$. Hence $E_0 =\frac{1}{6}\cir(1,1,1,1,1,1), E_3=\frac{1}{6}\cir(1,-1,1,-1,1,-1), E_1'= \frac{1}{6}\cir(2,1,-1,-2,-1,1), E_2'=\frac{1}{6}\cir(2,-1,-1,2,-1,-1)$.

    $S$ may then be used to form unitary and paraunitary matrices with real coefficients.
    \end{example}

The same process may  be applied in general to the normalised Fourier $n\ti n$ matrix to obtain COSI sets with real coefficients.  

\quad

By Propositions \ref{idem} and \ref{idem1}, paraunitary and unitary matrices of size $kn\ti kn$ are designed  from a COSI set $\{E_1,E_2, \ldots, E_k\}$ of $k$ elements of size $n\ti n$. 
The following constructs paraunitary and unitary matrices of size $n\ti n$ from a COSI set of size $n\ti n$. 
          \begin{proposition}\label{newt}\cite{tedbarry}. 
            Let $\{E_1, E_2, \ldots, E_k\}$ be a COSI  set.
            
            (i) Define $U(z) = \di\sum_{j=1}^k\pm E_jz^{t_j}$. Then $U(z)U^*(z^{-1})=I$.
            
            (ii) Let ${\bf{z}}=(z_1,z_2,\ldots, z_k)$ and  $U({\bf{z}}) = \di\sum_{j=1}^k E_jz_j$. Then $U({\bf{z}})U^*({\bf{z}^{-1}})=I$. 

            (iii)  Define $U(z) = \di\sum_{j=1}^k e^{i\theta_j} E_jz^{t_j}$. Then $U(z)U^*(z^{-1})=I$.
\end{proposition}
           When the $z$ is replaced by an element of modulus $1$ in
         part   (i) of Proposition \ref{newt},  a unitary matrix is
           obtained. Other versions of Proposition \ref{newt} may be
           formulated, for example by letting some  of the $z_j$ in
           part (ii) of Proposition \ref{newt} be equal. 

     Using COSI sets to design  paraunitary matrices is developed
     further in section \ref{para}. Using COSI sets to construct types of Hadamard matrices is developed in section \ref{hada}. In these sections, the COSI methods are combined with the designs methods of section \ref{dita}.

\subsubsection{Symmetric unitary matrix and further paraunitary matrices using COSI}  
 $U$ is a symmetric unitary matrix if and only if $U=(I - 2 E)$ where
$E$ is a (symmetric) idempotent,  see  \cite{ted1}, Proposition 8. 
This gives  the method for  constructing  a unitary symmetric matrix  from {\em any} idempotent.

Let $E$ be a symmetric idempotent. Then $\{E,I-E\}$ is a COSI set. 
Define  $U=(I-2E)$ which is then a unitary symmetric matrix and every
symmetric unitary matrix is of this form.  Note $(I-2E)E= -E, \, (I- 2E)(I-E) = I-E$ and thus
$(I-2E)$ has eigenvalue $-1$ occurring to multiplicity equal to $\rank
E$ and has eigenvalue $1$ occurring to multiplicity equal to $\rank
(I-E)$. 

The  renowned building blocks for 1D (one dimensional) paraunitary matrices over $\C$ due to
Belevitch and Vaidyanathan as described in
\cite{vaid} are constructed from a complete orthogonal idempotent set
of two elements in this manner. 

The requirement  that $U$ be of a particular type of symmetric matrix  can be  more difficult. Now  $H$ is a symmetric Hadamard matrix if and only if $U=\frac{1}{\sqrt{n}}H$ is a symmetric unitary matrix if and only if  this $U$ has a form $(I-2E)$ for a symmetric idempotent $E$. Thus a search for symmetric Hadamard matrices  could begin with a search for such idempotents. 

Suppose $U$ is any unitary matrix. Then its columns $\{u_1, u_2, \ldots, u_n\}$ give rise to the COSI set $\{E_1,E_2,\ldots, E_n\}$ with $E_i=u_iu_i^*$. Some of the $E_i$ may be combined to form different  COSI sets:  $\sum_{i=1}^kE_{j_i}$ is also a symmetric idempotent, with $J=\{j_1,j_2,\ldots ,j_k\} \subset  \{1,2,\ldots,n\}$,  and this idempotent is orthogonal to each $\{E_j | j\not\in J\}$ or any idempotent formed in this way from $\{E_j | j\not\in J\}$.

\begin{example} Let $E=\frac{1}{3}\left(\begin{smallmatrix} 1 & 1& 1 \\ 1&1&1\\ 1&1&1 \end{smallmatrix}\right), F= \frac{1}{3}\left(\begin{smallmatrix} 1 & \om & \om^2 \\ \om^2 &1 & \om \\ \om&\om^2&1 \end{smallmatrix}\right)$ where $\om$ is a primitive $3^{rd}$ root of unity. Then $\{E,F\}$ are idempotents and $U=(I-2E), V=(I -2F)$ are unitary matrices. Note $K=\sqrt{3}U$ satisfies $KK^*= K^2 = 3I_3$ but is not a Hadamard matrix. Also $UV = VU$ as $E,F$ are orthogonal.
\end{example}

Infinite series of symmetric unitary matrices may be obtained as illustrated in the following example.

\begin{example}\label{nott} Let $\{E_1,E_2\}$ be a COSI set. Then $U=\begin{smatrix} E_1 & E_2 \\ E_2&E_1\end{smatrix}$ is a symmetric unitary matrix. Thus $F_1=\frac{1}{2}(I-U), F_2=\frac{1}{2}(I+U)$ is an orthogonal set of idempotents and so $U_1=\begin{smatrix} F_1 & F_2 \\ F_2&F_1\end{smatrix}$ is a symmetric unitary matrix. Then $\{\frac{1}{2}(I-U_1), \frac{1}{2}(I+U_1)\}$ is a COSI set with which to form symmetric unitary matrices. This process may be continued to produce an infinite series of symmetric unitary matrices. 

  Initial choices for $\{E_1,E_2\}$ include $ \{E_1= \frac{1}{{2}} \begin{smatrix}1&1 \\1&1\end{smatrix}, E_2 = \frac{1}{{2}} \begin{smatrix}1&-1 \\-1&1\end{smatrix}\}$ and  $ \{E_1= \frac{1}{{2}} \begin{smatrix}1&i \\-i&1\end{smatrix}, E_2 = \frac{1}{{2}} \begin{smatrix}1&-i \\i&1\end{smatrix}\}$.
The $\{E_1,E_2\}$ can be of any size and not just $2\ti 2$ matrices and any COSI set may be used initially. 
\end{example}




 
          \begin{example}\label{above} 
          Let $U=\frac{1}{\sqrt{2}} \begin{smatrix} 1 & -1 \\ 1 &
            1 \end{smatrix} = \begin{pmatrix} u_1
            & u_2 \end{pmatrix}$.
          Define  $E_1 = u_1u_1^* =
          \frac{1}{2} \begin{smatrix} 1 & 1 \\ 1 & 1 \end{smatrix},
           E_2 = u_2u_2^*= \frac{1}{2}\begin{smatrix} 1 & -1 \\ -1 &
            1 \end{smatrix}$. 
            Then $U(z) = E_1z^i + E_2z^j$ is a paraunitary
          matrix. $U(z)$ has real entries and is symmetric in that
          $U(z)^*=U^*(z^{-1})=U(z^{-1})$.  Multiplying any two of the  form $U(z)$ using the same COSI set gives another of this form. However different COSI sets may be used to form paraunitary of the form $U(z)$ and these may be combined to give different types of paraunitary matrices. 
\end{example}
\begin{example} This gives an example of the design of a filter bank from COSI sets. 
 A unitary real $2\ti 2$ matrix is of the form $\begin{pmatrix} \cos
   \theta & \sin \theta \\ - \sin \theta & \cos
   \theta \end{pmatrix}$. The above matrix $U$ of Example \ref{above} is of this form  where $\theta = -\frac{\pi}{4}$. Define $E_1= \begin{pmatrix} \cos^2 \theta & -\cos \theta \sin  \theta \\ -\sin \theta \cos \theta & \sin^2 \theta \end{pmatrix}, E_2= \begin{pmatrix} \sin^2 \theta & \cos \theta \sin  \theta \\ \sin \theta \cos \theta & \cos^2 \theta \end{pmatrix}$. Then $\{E_1,E_2\}$ is a COSI set and $U(z) = E_1z^i + E_2z^j$ is a paraunitary matrix. Different $U(z)$ are obtained by taking different values of $\theta$ and these can then be used  to design  other paraunitary matrices of different forms. Paraunitary matrices of the type $A_0 + A_1z + \ldots + A_{2n-1}z^{2n-1}$ are  obtained with real coefficients.    From this a 2-channel filter bank with $n$ taps may be constructed.
\end{example}

\subsubsection{Group ring} The primitive central idempotents, see \cite{sehgal},  of the group ring $\C G$ form a complete orthogonal set of idempotents and these can be realised  as a COSI set in $\C_{n\ti n}$ where $n$ is the order of the group $G$. Interesting unitary and paraunitary matrices may be formed from the group ring $\C G$ of a finite group. 

          The unitary and paraunitary matrices formed  have rational coefficients when $G=S_n$, the symmetric group on $n$ letters, and have real coefficients when $G=D_n$ the dihedral group of order $2n$. Central primitive idempotents may also be combined to give a COSI with real entries 
       as    the idempotents occur in types of conjugate pairs. 
                  Some examples may be found in \cite{tedbarry}. The group ring aspects need to  be investigated further; some ideas for this paper occurred while looking at COSI sets in group rings.
          
\subsection{Di\c{t}\u{a} type}\label{dita} The following
constructions were initiated by  Di\c{t}\u{a}, \cite{dita}, 
and were essentially designed  in order to build  Hadamard matrices from
lower order Hadamard matrices.  They have been rediscovered in various
forms a number of times including by us. The original definition
involved square matrices only and  here it is generalised to work for
non-square matrices and with two `sides', left and right.  







\begin{Definition}\label{tangle}(Di\c{t}\u{a} \cite{dita})
  Let $\{A_1, A_2, \ldots, A_k\}$ be $m\ti n$ matrices and let $U=(u_{ij})$ be a $k\ti k$ matrix. Define {\em the  left matrix  tangle product} of $\{A_1, A_2, \ldots, A_k\}$ relative to $U$ to be the $mk\ti nk$ matrix 

  $ \begin{pmatrix} A_1u_{11} & A_2u_{12} & \ldots & A_ku_{1k} \\ A_1u_{21} & A_2u_{22} & \ldots & A_ku_{2k} \\ \vdots & \vdots & \ldots & \vdots \\ A_1u_{k1} & A_2u_{k2} & \ldots & A_ku_{kk} \end{pmatrix} $

  and {\em the right matrix tangle product} of $\{A_1, A_2, \ldots, A_k\}$ relative to $U$ to be the $mk\ti nk$ matrix 

  $ \begin{pmatrix} A_1u_{11} & A_1u_{12} & \ldots & A_1u_{1k} \\ A_2u_{21} & A_2u_{22} & \ldots & A_2u_{2k} \\ \vdots & \vdots & \ldots & \vdots \\ A_ku_{k1} & A_ku_{k2} & \ldots & A_ku_{kk} \end{pmatrix} $.

  The  notation $( U;A_1,A_2,\ldots, A_k)$ is used for  the left matrix tangle  product and 
$(A_1,A_2,\ldots, A_k;U)$ is used for  the right matrix   tangle product.  From the context it will often be clear which (left or right) matrix tangle  product is being used and in this case the term {\em matrix tangle product} is  utilised.
\end{Definition}
The Di\c{t}\u{a} construction as in \cite{dita,hos,craigen} is given as a left matrix tangle product with square matrices. 
The right tangle product is not equal to the left tangle product but
$(A_1,A_2,\ldots, A_k;U) = (U\T;A_1\T,A_2\T, \ldots, A_k\T)\T$ for
square matrices. It is convenient here for applications to have both
left and right constructions and also for constructions  when the
matrix $U$ is not square, see Definition \ref{tangle1} below. 

A  generalised version of this construction has also been used  but
this is not needed  here. The present constructions are used with a
view to designing non-separable matrices  in particular. 

Definition \ref{dita} can be generalised as follows to the case where $U$ is not square but has size either
$k\ti n$ or $n\ti k$ where $k$ is the number of matrices to be entangled; this requires the left or right matrix tangle product definitions. 
\begin{Definition}\label{tangle1}
(i)   Let $\{A_1, A_2, \ldots, A_k\}$ be $m\ti n$ matrices and let $U=(u_{ij})$ be a $t\ti k$ matrix. Define {\em the  left matrix  tangle product of $\{A_1, A_2, \ldots, A_k\}$ relative to $U$} to be the $tm\ti nk$ matrix 

  $$ \begin{pmatrix} A_1u_{11} & A_2u_{12} & \ldots & A_ku_{1k} \\ A_1u_{21} & A_2u_{22} & \ldots & A_ku_{2k} \\ \vdots & \vdots & \ldots & \vdots \\ A_1u_{t1} & A_2u_{t2} & \ldots & A_ku_{tk} \end{pmatrix} $$

 (ii) Let  $\{A_1, A_2, \ldots, A_k\}$ be $m\ti n$ matrices and let $U=(u_{ij})$ be a $k\ti t$ matrix. Define {\em the right matrix  tangle product of $\{A_1, A_2, \ldots, A_k\}$ relative to $U$} to be the $mk\ti nk$ matrix 

  $$ \begin{pmatrix} A_1u_{11} & A_1u_{12} & \ldots & A_1u_{1t} \\ A_2u_{21} & A_2u_{22} & \ldots & A_2u_{2t} \\ \vdots & \vdots & \ldots & \vdots \\ A_ku_{k1} & A_ku_{k2} & \ldots & A_ku_{kt} \end{pmatrix} $$

 The  notation $(U; A_1,A_2,\ldots, A_k )$ is used for  the left  matrix tangle  product and 
$(A_1,A_2,\ldots, A_k;U)$ is used for  the right  matrix tangle product. From the context it may  be clear which (left or right) tangle  product is being used and in this case the term {\em matrix tangle product} is utilised.
\end{Definition}

The matrix tangle product depends on the order of the $A_i$ and
different tangle products are obtained from  different permutations of
the $A_i$  - `different permutations' should take into account that
some of the $A_i$ may be the same.   This can be particularly useful
in designing series of different entangled  matrices with desired properties.

If all the $A_i=A$ are the same then the matrix tangle   product is the matrix tensor\footnote{Matrix tensor product is often called {\em Kronecker product}. See however \cite{zeh} for discussion on this name.} 
product $U\otimes A$.
The direct sum of matrices is also a very special matrix tangle product as  $\begin{pmatrix} A & 0\\0 & B \end{pmatrix} = (I_2;A,B)$.

Say $U$ is the {\em shuffler} matrix and say $\{A_1,A_2, \ldots, A_k\}$ are the {\em tangle} matrices of the matrix tangle product $(U; A_1,A_2, \ldots, A_k)$ or of $(A_1,A_2, \ldots, A_k;U)$ depending on which, left or right, matrix tangle product is under consideration. Suppose now an $m\ti n$ matrix $U$ is to be a shuffler matrix of a matrix tangle product. Then either $m$ or $n$ matrices are required for the tangles {\em but} they need not all be different. If they are all the same and equal to $A$ then the tensor product   $U\otimes A$ is obtained which is an $mt\ti nq$ matrix when $A$ is $t\ti q$. If less than $n$ or $m$ different matrices are to be used as tangles then these are repeated until $m$ or $n$ matrices are obtained as appropriate.  


The matrix tangle product may be square even though neither the tangles nor the shuffler are square. For example if $\{A,B\}$ are $2\ti 3$ matrices and $U$ is a $3\ti 2$ matrix then $(U;A,B)$ is a $6\ti 6$ matrix. In general if $\{A_1,A_2, \ldots, A_k\}$ are $k\ti t$ matrices  and $U$ is $t\ti k$ then   $(U;A_1, A_2, \ldots, A_k)$ is a $kt\ti kt$ matrix; if  $\{A_1,A_2, \ldots, A_k\}$ are $t\ti k$ matrices  and $U$ is $k\ti t$ then   $(A_1, A_2, \ldots, A_k;U)$ is a $kt\ti kt$ matrix. 


The matrix tangle  product is not a matrix tensor  product unless
there is a fixed $A$ such that 
$A_i=\al_iA$ for some $\al_i$. In this situation  
$(U;\al_1A, \al_2A, \ldots, \al_kA)= (U';A,A,\ldots,A) = U'\otimes A$ where $U'$ is obtained from $U$ by multiplying rows or columns of $U$ by appropriate $\al_i$.  

The matrix tangle product has some linearity:
\begin{itemize} \item $\al(U;A_1, A_2, \ldots, A_k)=(U;\al A_1,\al A_2, \ldots, \al A_k)= (\al U; A_1,A_2, \ldots, A_k)$.   \item $ (U+V;A_1,A_2,  \ldots, A_k) = 
  (U;A_1,A_2, \ldots, A_k) + (V;A_1,A_2, \ldots, A_k)$. \item $(U;A_1, A_2, \ldots, A_k)+ (U;B_1, B_2, \ldots, B_k) = (U;A_1+B_1, A_2+B_2, \ldots, A_k+B_k)$.
   
\end{itemize}

Similar results hold  for the right matrix tangle product.

Note however for example that $(U;A_1+A_1',A_2)$ is not in general the same as $(U;A_1,A_2)+(U;A_1',A_2)$. 

\subsection{Determinant}\label{det}
The determinant value  of a matrix tangle product of square matrices
 in terms of the constituents is interesting and valuable. It can be obtained in terms of  the determinants of the tangles
 and shuffler, see Proposition
 \ref{tom} below. However the spectrum does not have a relationship with the
 spectrums of the constituents,  as happens for  a matrix tensor
 product, as the process produces an {\em entangled} matrix in
 general. 

Let  $T=(U;A_1,A_2,\ldots, A_k)$.  It is of interest  to know the
value of $\det T = |T|$ when the $A_i$ and $U$ are square matrices. It
is given in terms of the determinants of the constituents as follows. 

\begin{proposition}\label{tom} Let $T=(U;A_1,A_2,\ldots, A_k)$ where $U$ is a $k\ti k$ matrix and the $A_i$ are $n\ti n$ matrices. Then $|T| =  |A_1||A_2|...|A_k||U|^n$.
\end{proposition}
\begin{proof} This can be shown using results on determinants of block
  matrices as for example in \cite{powell}. Alternatively a direct
  proof may be given by applying the techniques used when working with
  proofs of  determinants  on block matrices. 
  Proceed inductively as follows. Let $U= (\al_{ij})$. 
  If $A_1=0$ or if  all $\al_{1i}=0$ the result is clear. We can
  assume   we  can assume $\al_{1i}\neq 0$ for some
  $\al_i$ and hence by block operations we can assume $\al_{11} \neq 0$.  
  Then apply block operations on $T$
  to reduce the first column of blocks to the form $\begin{pmatrix}
    \al_{11} A_1 \\ 0 \\ \vdots \\ 0 \end{pmatrix}$;  these block
  operations do not alter the value of the determinant. Then $|T| =
  \det(\al_{11}A_1) \ti |B|$ where $B$ is a similar matrix to $T$ but
  of one block size smaller; induction may then be applied.   
\end{proof}

A similar result holds for the right matrix  tangle product. 

This property is particularly useful in applications, see for example
Section \ref{spacetime}.  Proposition \ref{tom}  generalises the determinant value of a matrix  tensor product --  if all the $A_i$ are the same, $A_i=A$,  then $|T|= |A|^k|U|^n$ and  $T= U\otimes A$.
For example let $\{A,B\}$ be $n\ti n$ matrices and let $U$ be of size $2\ti 2$.
Then $T=(U;A,B) $ has $ |T|= |A||B||U|^n$. 
The determinantal property of tensor
products which is a special case of Proposition \ref{tom} is in itself very
useful in many areas. 


Finding the eigenvalues of a matrix  tangle product is difficult and
no formula in terms of the eigenvalues of the constituents exists.  The eigenvalues of a matrix tangle product are `entangled'.  

 \subsection{Preserved properties}

Which properties of the shuffler and tangles of a matrix tangle product are preserved? 
Let $\mathcal{P}$ be a property of a matrix,  such as for example being unitary or invertible. 
Say $M\in \mathcal{P}$ if and only if $M$ has this property $\mathcal{P}$.  
If for any $G=(U;A_1, A_2, \ldots,A_k)$ with $A_i\in \mathcal{P}$ for $i=1,2,\ldots,k$ and $U\in \mathcal{P}$,  implies that $G\in \mathcal{P}$ then say the matrix tangle product {\em preserves} $\mathcal{P}$.   
\begin{itemize} \item The property of being a unitary matrix is
  preserved.  \item The property of being an invertible matrix  is
  preserved. \item The property of being a paraunitary matrix is preserved.
\item The property of being a normal matrix is not preserved. \item
  The property of being a symmetric matrix  is not preserved. \item
  The property of being a Hadamard matrix  is preserved. \end{itemize}


The preserved properties are stated as Propositions in the following
subsections \ref{unitary}, \ref{para} and \ref{hada}. These
subsections derive applications, constructions and designs.

\subsection{Unitary}\label{unitary} 
\begin{proposition}\label{unit} Let $\{A_1, A_2, \ldots, A_k\}$ be
  $m\ti m$ unitary matrices and  let $U=(u_{ij})$ be a unitary $k\ti
  k$ matrix. Then \\ 
$ \begin{pmatrix} A_1u_{11} & A_2u_{12} & \ldots & A_ku_{1k} \\ A_1u_{21} &
    A_2u_{22} & \ldots & A_ku_{2k} \\ \vdots & \vdots & \vdots & \vdots \\ A_1u_{k1} & A_2u_{k2} & \ldots & A_ku_{kk} \end{pmatrix} $
and $ \begin{pmatrix} A_1u_{11} & A_1u_{12} & \ldots & A_1u_{1k} \\ A_2u_{21} &
    A_2u_{22} & \ldots & A_2u_{2k} \\ \vdots & \vdots & \vdots & \vdots \\ A_ku_{k1} & A_ku_{k2} & \ldots & A_ku_{kk} \end{pmatrix} $
are  unitary $mk\ti mk $ matrices.
\end{proposition}

Thus the matrix tangle products of unitary matrices are  unitary
matrices. Section \ref{gyt} also constructs unitary matrices  from
COSI (complete orthogonal symmetric idempotent)  sets. This greatly
expands the pools of unitary matrices available for various
purposes. Entangled matrices are often required and this condition can
be realised by these constructions.

A matrix is unitary if and only if its rows or columns form an
orthonormal basis and thus new orthonormal bases are constructed when
a new unitary matrix is constructed. 


\begin{example}{Pauli unitary matrices as builders for higher order matrices}

  Applying the process to the Pauli matrices $\sigma_x =\begin{ssmatrix}0& 1 \\ 1 & 0\end{ssmatrix}, \sigma_y=\begin{ssmatrix} 0 & i \\ -i &0 \end{ssmatrix}, \sigma_z=\begin{smatrix} 1 & 0 \\ 0 & -1 \end{smatrix}$ gives interesting entangled unitary matrices. The following six $4\ti 4$ matrices are obtained when each of the matrices is used (once) as a tangle or as a shuffler:

$(\s_z;\s_x,\s_y), (\s_z;\s_y,\s_x), (\s_y;\s_x,\s_z), (\s_y;\s_z,\s_x), (\s_x;\s_z,\s_y), (\s_x;\s_y,\s_z)$

Other $4\ti 4$ unitary matrices may be formed from $\{\s_x,\s_y,\s_y\}$;  some are tensor products such as $(\s_x,\s_x;\s_z)$ and ones are like $(\s_x,\s_y;\s_x)$ where a matrix appears as both a tangle and the shuffler.
Taking two of these $4\ti 4$ unitary matrices as tangles and using one of $\{\s_x,\s_y,\s_z\}$ as a shuffler produces an $8\ti 8$ unitary matrix in which the Pauli matrices are constituents and entangled.

This process may be continued to produce $2^n\ti 2^n$ unitary entangled matrices from the Pauli matrices. The significance of these needs to be explored. 
\end{example} \begin{example}{Real unitary}
  
Start with the following real $2\ti 2$ matrices 
$\begin{ssmatrix} -1 & 1 \\ 1 & 1 \end{ssmatrix}, \begin{ssmatrix} 1 & -1\\ 1 & 1 \end{ssmatrix},\begin{ssmatrix} 1 & 1 \\ -1 & 1 \end{ssmatrix}, \begin{ssmatrix} 1 & 1 \\ 1 & -1\end{ssmatrix}$ from which to build new matrices. Make these  unitary by dividing by $\sqrt{2}$ and then unitary matrices are  built by the construction methods. 

  The following real unitary (orthogonal) matrices
  $\begin{ssmatrix} \cos \theta & \sin \theta \\ -\sin \theta & \cos \theta \end{ssmatrix}$ are often used in practice. 
  Different $\theta$ may be used from which real $2^n\ti 2^n$ real orthogonal matrices are  built.

  Now $A_i=e^{i\theta_i}$ are $1\ti 1$ unitary  matrices.
Let $U$ be a $k\ti k$ unitary  matrix. \\ Then $(U;A_1,A_2,\ldots, A_k),(A_1,A_2,\ldots, A_k;U)$ are also unitary $k\ti k$ matrices. 
\end{example}
  
  \begin{example}{Unbiased bases example} \label{unbiased}
    
\begin{itemize} \item   Let $U = \frac{1}{\sqrt{2}}\begin{ssmatrix} 1 & 1 \\ 1 & -1 \end{ssmatrix}$ and let $ A= (1), B= (i)$. Then $\{U,A,B\}$ are unitary matrices. Now $(A,B;U)$ is a unitary matrix $G=\frac{1}{\sqrt{2}}\begin{smatrix} 1 & 1 \\ i & -i \end{smatrix}$. Then $\{U,G,I_2\}$ constitute three matrices consisting of mutual unbiased bases for $\C^2$.
\item   Let $U=\frac{1}{\sqrt{3}}\begin{ssmatrix} 1 & 1 &1 \\ 1 & \om & \om^2 \\ 1 & \om^2 & \om \end{ssmatrix}$. Let $A=(1), B= (\om), C= (\om)$ and form
  $U_1 = (A,B,C;U)$. Let $A= (1), B= (\om^2), C= (\om^2)$ and form
  $U_2 = (A,B,C;U)$.  Then $\{U,U_1,U_2,I_3\}$ are 4  matrices consisting of mutually unbiased bases for $\C^3$.  
\end{itemize}
\end{example} 

  \subsection{Additional Paraunitary concepts}\label{para} 
  Paraunitary matrices are fundamental  in signal
  processing and the concept of a paraunitary matrix
  plays an important  role in the research area of multirate 
filterbanks and  wavelets.  
In the polyphase domain, the synthesis matrix of an orthogonal filter
bank is a paraunitary matrix;  
a Filter Bank is orthogonal if its polyphase matrix is paraunitary, see \cite{strang}. Thus designing an orthogonal filter bank is {\em equivalent} to designing a paraunitary matrix. The book \cite{strang}, Chapters 4-6,  makes the design of paraunitary matrices a primary aim. 
Designing  {\em entangled} paraunitary matrices is often a
requirement and has been a difficult task.

 The literature is huge and expanding rapidly; of particular note is   \cite{zhou}, where further background and many references may be found.  
From the literature: ``Designing nonseparable multidimensional orthogonal filter banks is a challenging task.''; ``Multirate filter banks give the structure required to generate important cases of wavelets and the wavelet transform.''; `` In filter bank literature  the terms  orthogonality, paraunitary and lossless are often used interchangeably.''
``Paraunitryness is a necessary and sufficient condition for wavelet orthogonality.''
 ``Designing an orthogonal filter bank is equivalent to designing a paraunitary matrix.''

Paraunitary matrices are constructed using COSI sets by methods of  Propositions \ref{idem1} and \ref{newt}, see Section \ref{gyt}; paraunitary matrices which are symmetric may be built with this method. 

`Being a paraunitary matrix' is a property preserved by matrix tangle products.
\begin{proposition}\label{het}
Let $\{A_1, A_2, \ldots, A_k\}$ be $m\ti m$ paraunitary matrices and  let $U=(u_{ij})$ be a paraunitary $k\ti k$ matrix. Then

  $ \begin{pmatrix} A_1u_{11} & A_2u_{12} & \ldots & A_ku_{1k} \\ A_1u_{21} &
    A_2u_{22} & \ldots & A_ku_{2k} \\ \vdots & \vdots & \vdots & \vdots \\ A_1u_{k1} & A_2u_{k2} & \ldots & A_ku_{kk} \end{pmatrix} $ and 
$ \begin{pmatrix} A_1u_{11} & A_1u_{12} & \ldots & A_1u_{1k} \\ A_2u_{21} &
    A_2u_{22} & \ldots & A_k2u_{2k} \\ \vdots & \vdots & \vdots & \vdots \\ A_ku_{k1} & A_ku_{k2} & \ldots & A_ku_{kk} \end{pmatrix} $  

are paraunitary  $mk\ti mk $ matrix in the union of the variables in $\{A_1,A_2,\ldots, A_k, U\}$.
  
\end{proposition}

The constructions  in Propositions \ref{idem1}, \ref{newt} and \ref{het} may be combined. 
Building blocks for paraunitary matrices are available; these are not
tensor products and are entangled in general. 
The shuffler itself may be a unitary matrix as may any of the tangles.  Examples are given in \cite{tedbarry} where a more restricted tangle definition is given. Although the systems here give building blocks for multidimensional paraunitary matrices, it is not claimed that every multidimensional paraunitary matrix is built in this way although many such are built in this manner.  
The  renowned building blocks for 1D paraunitary matrices over $\C$ due to
Belevitch and Vaidyanathan as described in
\cite{vaid} are constructed from a complete orthogonal idempotent set of two elements. 

Now $A_i=z_i$ are $1\ti 1$ paraunitary  matrices.  
Let $P$ be a $k\ti k$ paraunitary  matrix. Then $G=(P;A_1,A_2,\ldots, A_k)$ is a paraunitary $k\ti k$ matrix in the union of the variables in $P$ and $\{z_1,z_2,\ldots, z_k\}$.

By replacing the variables by elements of modulus $1$ in a paraunitary matrix, a unitary matrix is obtained. Constructing paraunitary matrices leads to the construction of unitary matrices.

\subsection{Hadamard $\leftrightarrow$ Unitary}\label{hada}   
$H$ is  a real Hadamard $n\ti n$ matrix if its entries are elements of
modulus $1$ and $HH^*=nI_n$. A Hadamard matrix of type $H(n,p)$ is a 
matrix in which each element of $H(n,p)$ is a $p^{th}$ root of $1$ and
$H(n,p)H(n,p)^*=nI_n$. A $H(n,2) $  matrix is a real Hadamard matrix
$n \times n$ matrix. It is known that the Di\c{t}\u{a}
construction preserves Hadamard matrices, \cite{dita,hos,craigen}.

\begin{proposition}\label{unit1}\cite{dita} Let $\{A_1, A_2, \ldots,
  A_k\}$ be $m\ti m$ Hadamard  matrices and let $U=(u_{ij})$ be a
  Hadamard  $k\ti k$ matrix. Then
  
(i) \,  $\left(\begin{smallmatrix} A_1u_{11} & A_2u_{12} & \ldots & A_ku_{1k} \\ A_1u_{21} &
    A_2u_{22} & \ldots & A_ku_{2k} \\ \vdots & \vdots & \vdots & \vdots \\ A_1u_{k1} & A_2u_{k2} & \ldots & A_ku_{kk} \end{smallmatrix}\right) $
is a Hadamard $km\ti km $ matrix. If the $A_i$ and $U$ have entries which are $n^{th}$ roots of $1$ 
then this matrix has entries which are $n^{th}$ roots of $1$.
  
  (ii) \, $\left( \begin{smallmatrix} A_1u_{11} & A_1u_{12} & \ldots & A_1u_{1k} \\ A_2u_{21} &
    A_2u_{22} & \ldots & A_2u_{2k} \\ \vdots & \vdots & \vdots & \vdots \\ A_ku_{k1} & A_ku_{k2} & \ldots & A_ku_{kk} \end{smallmatrix}\right) $
is a Hadamard $km\ti km $ matrix. If the $A_i$ and $U$ have entries which are $n^{th}$ roots of $1$ 
then this matrix  has entries which are $n^{th}$ roots of $1$.
\end{proposition}



The Di\c{t}\u{a} product has been used in a number of papers to
construct Hadamard matrices from lower order Hadamard matrices, see
for example \cite{dita} itself,  and also \cite{hos} and \cite{craigen}. Hadamard matrices have been also  constructed in section \ref{gyt} by the COSI method. 

Now $A_i=e^{i\theta_i}$ are $1\ti 1$ Hadamard matrices. Say $H$ is a $H(n,p)$ matrix if it has size $n$ and entries which are $p^{th}$ roots of $1$. 
Let $H$ be a $k\ti k$ Hadamard matrix. Then $G=(H;A_1,A_2,\ldots, A_k)$ is a Hadamard matrix. If $H=H(k,p)$ and $\{A_i = A_{i}(1,p)\}$ then $G$ is a $G(k,p)$ matrix. If $H=H(k,p)$ and $A_i=A_i(1,n_i)$ then $G$ is a $G(k,s)$ matrix where
$s =\lcm(p,n_1, n_2, \ldots, n_k)$.  

Symmetric Hadamard matrices are {\em Type II} matrices; the definition and further information on Type II matrices  may be found in  \cite{hos}  and  the many references therein. ``Type II matrices were introduced explicitly in the study of {\em spin} models.'' The following construction is similar to that formulated in for example
\cite{craigen} but is a useful way with which to look at the formulation of symmetric Hadamard matrices. 
\begin{construction}{Construct symmetric Hadamard matrices.}\label{hadsymm}
  
{\em Let $H$ be a Hadamard matrix of type $H(n,p)$. 
Let $G$ be the corresponding unitary matrix, that is $G=\frac{1}{\sqrt{n}} H$. The columns $\{u_1,u_2, \ldots, u_n\}$ of $G$ form an orthonormal basis for $\C_n$. Let $E_i=u_iu_i^*$. Then $\{E_1, E_2, \ldots, E_n\}$ is a COSI set, from which unitary $n^2 \ti n^2$ matrices may be formed as in section \ref{gyt}. In particular symmetric $n^2 \ti n^2$ matrices may be formed using the reverse circulant construction. These matrices have entries which are $\frac{1}{n}$ times a $p^{th}$ root of $1$ and so multiplying any of these matrices  by $n$ gives a symmetric $n^2\ti n^2$ Hadamard matrix which is a $H(n^2,p)$ matrix.}
\end{construction}
Starting from {\em any} Hadamard $H(n,p)$, Construction \ref{hadsymm} designs series of
Hadamard $H(n^2,p)$ matrices. These can be designed  to be symmetric
by using reverse circulant form. 
The process may then be continued to produce $H(n^{2^k},p)$, for
$k\geq 1$ Hadamard matrices going via unitary matrices.  By taking the reverse
circulant process at any stage of production the matrices produced are
symmetric. Only at the final stage need the reverse circulant process
be applied in order to design symmetric Hadamard matrices. 

It is also known, see for example \cite{craigen},  that a symmetric
$2n\ti 2n $ Hadamard symmetric matrices may be constructed from  $n\ti
n$ symmetric Hadamard matrices. The construction \ref{tomm}
below is similar but different and illustrates the niceness  of the
tangled product in general for designs.

(Recall:A Hadamard matrix $H$ is said to be of type $H(n,p)$ if
it is an $n\ti n$ Hadamard matrix and all its entries are $p^{th}$ roots of unity.) 

\begin{construction}\label{tomm}

  (i) Let $H$ be an $n\ti n$ Hadamard symmetric matrix and $U$ a $2\ti 2$ symmetric matrix. Then $(U;A,A\T), (U;A\T,A), (A,A\T;U),(A\T,A:U)$
  are symmetric  Hadamard $2n\ti 2n$ matrices.

  (ii) Let $H$ be an $n\ti n$ Hadamard symmetric  matrix of type
  $H(n,p)$ and $U$ a $2\ti 2$ symmetric  matrix. Then $(U;A,A\T), (U;A\T,A), (A,A\T;U),(A\T,A:U)$
  are symmetric  Hadamard $2n\ti 2n$ matrices of type $G(2n,p)$. More generally
  if $H$ is of type $H(n,p)$ and $U$ is of type $U(2,q)$ then $(U;A,A\T), (U;A\T,A), (A,A\T;U),(A\T,A:U)$ are of type $G(2n,s)$ where $s = \lcm(q,p)$.
  
\end{construction}



 The $n\ti n$ Fourier matrix is  a Hadamard $H(n, n)$ matrix.  

\begin{example} Let  $H= \begin{pmatrix} 1 & 1 & 1 \\ 1 & \om & \om^2 \\ 1 & \om^2 & \om \end{pmatrix}$ where $\om$ is a primitive third root of $1$. Then $G= \frac{1}{\sqrt{3}} H$ is a unitary matrix.  The columns of $G$  are $u_1 = \frac{1}{\sqrt{3}}(1,1,1)\T, u_2=\frac{1}{\sqrt{3}}(1,\om,\om^2)\T,
u_3 =  \frac{1}{\sqrt{3}}(1,\om^2,\om)\T$.
Then $\{E_1 = u_1u_1^* = \frac{1}{3}\begin{pmatrix} 1&1 & 1 \\ 1 & 1 & 1 \\ 1 &1 & 1\end{pmatrix}, E_2 = u_2u_2^* = \frac{1}{3}\begin{pmatrix} 1&\om^2 & \om \\ \om & 1 & \om^2 \\ \om^2 &\om & 1\end{pmatrix}, E_3 = u_3u_3^*=\frac{1}{3}\begin{pmatrix} 1&\om & \om^2 \\ \om^2 & 1 & \om \\ \om &\om^2 & 1\end{pmatrix}\}$ is a COSI set.
Thus $K= \begin{pmatrix} E_1 & E_2 & E_3 \\ E_2 & E_3 & E_1 \\ E_3&E_1& E_2 \end{pmatrix}$ is a symmetric unitary matrix and
$L=3K$ is a symmetric Hadamard $L(9,3)$ matrix.

\end{example}
     


   




\begin{example}:
$P=  \begin{ssmatrix} 1& 1\\ 1& -1 \end{ssmatrix},
Q=  \begin{ssmatrix} 1 &  i \\ i & 1 \end{ssmatrix}$. 
are Hadamard $H(2,4)$ matrices. Then $A=\frac{1}{\sqrt{2}}P, B= \frac{1}{\sqrt{2}}B$ are unitary matrices. Infinite series of unitary and Hadamard matrices may be built as follows.  Build $\{A,B\}$ relative to unitary $A$ and then build $\{A,B\}$ relative to unitary $B$ to obtain
Build 
$A_1 = (A,B;A)= \frac{1}{2} \begin{ssmatrix} 1&1& 1&i \\ 1& -1 & i& 1 \\ 1 & 1 & -1 & -i \\ 1& -1 & -i & 1\end{ssmatrix}, B_1= (A,B;B)=  \frac{1}{2}\begin{ssmatrix}  1 & 1 & i & -1 \\ 1& -1 & -1 & i \\ i & i & 1& i \\ i & -i & i & 1 \end{ssmatrix}$. Other options for $A_1,B_1$ are $A_1=(A;A,B), B_1= (B;A,B)$ but also others such as swapping $A,B$ around. These are $4\ti 4$ matrices and $2 A_1, 2 B_1$ are Hadamard $H(4,4)$ matrices.  

Build $(A_1,B_1;A), (A_1,B_1;B)$ to get unitary $8\ti 8$ matrices with entries $\pm 1, \pm i$ and from these get $H(8,4)$ matrices. Build $(A_1, B_1;A_1), (A_1,B_1;B_1) $ to get $H(16,4)$ matrices. The process may be continued in many different directions.  

\end{example}










\subsubsection{Skew Hadamard}
A Hadamard $n\ti n$ matrix is a {\em skew Hadamard} provided $H = I_n+U$
with $U^*=-U$. If interested in real Hadamard matrices then it is required
that $U^*=U\T = -U$.
This implies $H+H^* = 2I_n$. The Di\c{t}\u{a} product may be used to produce skew $2n\ti
2n$ Hadamard matrices from a skew $n\ti n$ Hadamard matrix.
Skew Hadamard matrices are used in a number of areas including for the
construction of orthogonal designs. 

\begin{construction}\label{ditta} Let $A$ be an $n\ti n$ skew Hadamard matrix and let $U$ be a $2\ti 2$ skew symmetric Hadamard matrix. Then $(U;A,A\T), (U;A\T,A), (A,A\T;U), (A\T,A;U)$  are skew symmetric $2n\ti 2n $ matrices.
\end{construction}
The known   method, see for example \cite{craigen},  for producing a
$2n\ti 2n $ skew symmetric matrix from a $n\ti n$ skew symmetric
matrix is a special case of Construction \ref{ditta} above. 

The Construction \ref{ditta} works for general Hadamard skew matrices
over $\C$.
 
The $2\ti 2$ skew Hadamard real matrix used initially  could  be $\begin{smatrix} 1 & -1 \\ 1 & 1 \end{smatrix}$ or  $\begin{smatrix}1 & 1 \\ -1 & 1 \end{smatrix}$ or similar. A skew Hadamard $2\ti 2$ matrix over $\C$ in addition are ones of the form $\begin{smatrix} 1 & \al \\ -\al^* & 1 \end{smatrix}$ where $|\al| = 1$.  

Suppose now $H=\left(\begin{smallmatrix} 1 & e^{i\al_1} & - e^{-i\al_2 }&
-e^{-i\al_3} \\ -e^{-i\al_1} & 1 & -e^{-i\al_4} & e^{i\al_5}
\\ e^{i\al_2}&e^{i\al_4}&1 & e^{i\al_6} \\ e^{i\al_3} & - e^{-i\al_5}
& -e^{-i\al_6} & 1 \end{smallmatrix}\right)$ is to be a Hadamard matrix; it
already has the skew condition, $H = I+U$ with $U^* = -U$. Then looking at $HH^*=
4I_n$ the following  (just three) conditions are obtained: (i) $-\al_2 + \al_4 =
-\al_5 -\al_3$; (ii) $\al_1-\al_4 = -\al_3 - \al_6$; (iii)
$\al_1+\al_5 = \al_6 -\al_2$. Solving this system of equations gives
$\al_4=\al_1 + \al_2, \al_5= -\al_1-\al_3, \al_6=\al_2-\al_3$ and $\al_1,\al_2,\al_3$ can have any value. This gives an infinite number of skew Hadamard (complex)  matrices. 
New infinite sets can be formed using Construction \ref{ditta}.

\begin{example} As an example require now that the $\{e^{i\al_j}\}$ be $n^{th}$ roots of $1$. Say for example $ \al_1= \frac{2\pi}{n}, \al_2 = \frac{4\pi}{n}, \al_3= \frac{6\pi}{n}$ and then $\al_4= \frac{6\pi}{n},\al_5= -\frac{8\pi}{n}, \al_6= -\frac{2\pi}{n}$.
\\ This gives the following skew Hadamard matrix
$\left(\begin{smallmatrix} 1& \om & -\om^{-2} & -\om^{-3} \\ -\om^{-1}&1 & -\om^{-3} & \om^{-4} \\ \om^2 & \om^3 & 1 & \om^{-1} \\ \om^3 & -\om^4 & -\om & 1 \end{smallmatrix}\right)$, where $\om = e^{i\frac{2\pi}{n}}$ is a primitive $n^{th}$ root of $1$.

Further taking $\om$ to be a primitive third root of unity, $\om^3=1$,  gives the skew Hadamard matrix
$\left(\begin{smallmatrix} 1& \om &  -\om & -1 \\ -\om^2 & 1 & -1 & \om^2 \\ \om^2 & 1 & 1 & \om^2 \\ 1& -\om & -\om & 1 \end{smallmatrix}\right) $. The entries are $6^{th}$ roots of unity, so this is a $H(4,6)$ matrix.
\end{example}

\quad


Infinite sequences of skew Hadamard real  matrices may be obtained by
starting out with a skew Hadamard matrix real matrix    $A$ and with 
$U= \begin{smatrix}1 & -1 \\ 1 & 1 \end{smatrix}$ or $U= \begin{smatrix}1 & 1 \\ -1 & 1 \end{smatrix}$.

Then  form $A_1$ which can be one of $(U;A,A\T),(U;A\T,A), (A,A\T;U),(A\T,A;U)$.

Replace $A$ by $A_1$ to form  $(U;A_1,A_1\T),(U;A_1\T,A_1),
(A_1,A_1\T;U),(A_1\T,A_1;U)$ which are skew Hadamard matrices; this
process may be continued. 

Let $A$ be a normalised $n\ti n$ Fourier matrix and $B$ a matrix obtained from $A$ by interchanging rows (or columns). Then both $A,B$ are unitary matrices. Let $C$ be any $2\ti 2$ unitary matrix. Then $(A,B;C)$ and $(B,A;C)$ are unitary $2n\ti 2n$ matrices. Let $A$ be a Hadamard matrix and $B$ any permutation of the rows of columns of $A$. Let $C$ be any $2\ti 2$ Hadamard matrix. Then $(A,B;C)$ and $(B,A;C)$ are Hadamard matrices. If $A$ is of type $H(n,q)$ and $C$ is of type $H(2,q)$ then type of $(A,B;C), (B,A;C)$ have a determined type.

  \begin{example} As an explicit example consider the following: \\ Let $A=\frac{1}{\sqrt{3}}\left(\begin{smallmatrix} 1 & 1 & 1 \\ 1 & \om & \om^2 \\ 1 &\om^2 & \om \end{smallmatrix}\right), B= \frac{1}{\sqrt{3}}\left(\begin{smallmatrix} 1 & 1 & 1 \\ 1 & \om & \om^2 \\ 1 &\om^2 & \om \end{smallmatrix}\right), C=\frac{1}{\sqrt{2}}\left(\begin{smallmatrix} 1 & 1 \\ 1& -1 \end{smallmatrix}\right)$, where $\om$ is a primitive $3^{rd}$ root of unity. 

  Then $(A,B;C), (B,A;C)$ are $6\ti 6 $ unitary matrices with entries which are $\al=\frac{1}{\sqrt{6}}$ times $6^{th}$ roots of unity and so $\al(A,B;C), \al(B,A;C)$ are Hadamard matrices with entries which are $6^{th}$ roots of unity.  

  This can also be played out for the discrete cosine and sine transforms. Let $A,B$ be discrete transforms and $C$ any $2\ti 2$ unitary matrix. 
  Then $\{(A,B;C), (B,A;C)\}$ are multidimensional transforms which are not matrix tensor products.
\end{example}

Hadamard matrices have been designed from matrix tensor products -- if $A,B$ are Hadamard matrices so is $A\otimes B$. Many formulations of Hadamard constructions are equivalent to matrix tensor product constructions. 

Thus tangle product generalises the matrix tensor  product method for constructing Hadamard matrices; the matrix tensor product method includes Sylvester's method. Sylvester's method for producing Walsh matrices starts out with $U= \left(\begin{smallmatrix} 1 & 1 \\ 1 & -1 \end{smallmatrix}\right)$ and goes to $\left(\begin{smallmatrix}A & A \\ A & -A\end{smallmatrix}\right)$ where $A$ has already been constructed; this is $A\otimes U$. A similar series may be obtained by starting out with for example beginning with the same or different initial $U$ and then producing $(A,B; U)$ from previously produced $A,B$.
  Indeed the $U$ could change at any stage. The Walsh-Hadamard transfer has uses in many areas and is formed using a matrix tensor product starting out with $\left(\begin{smallmatrix} 1 & 1 \\ 1 & -1 \end{smallmatrix}\right)$. Many variations on this may be obtained using matrix tangle products; for instance the related  matrices
   $\left(\begin{smallmatrix} -1 & 1 \\ 1 & 1 \end{smallmatrix}\right),  \left(\begin{smallmatrix} 1 & -1 \\ 1 & 1 \end{smallmatrix}\right),  \left(\begin{smallmatrix} 1 & 1 \\ -1 & 1 \end{smallmatrix}\right),  \left(\begin{smallmatrix} 1 & 1 \\ 1 & -1 \end{smallmatrix}\right)$ could be used and entangled. 

  Hadamard matrices can  also be designed from paraunitary matrices which themselves have been designed by orthogonal symmetric  complete sets of idempotents, see section \ref{gyt}. 
 
\subsection{Combine COSI and Di\c{t}\u{a} type}\label{joint}
Subsection \ref{gyt} devises COSI constructions and
subsection\ref{dita} initiates Di\c{t}\u{a} type constructions. The
two may be combined to derive further builders. The COSI construction
can be used  to construct unitary, paraunitary or Hadamard matrices
and these may then be used to construct matrix types using the
Di\c{t}\u{a} construction. On the other hand suppose a unitary matrix
is constructed by either method. Then the columns of the matrix may be
used to construct COSI sets from which further unitary, paraunitary or
other  entangled matrix types  can be constructed by the COSI method of section
\ref{gyt}. 

\begin{example}\label{three}

  {\em Let $U= \frac{1}{\sqrt{2}}\begin{ssmatrix}1 & -1 \\ 1 & 1 \end{ssmatrix}, A= \frac{1}{\sqrt{2}}\begin{ssmatrix} 1 & 1 \\ 1 & -1 \end{ssmatrix}, B= \frac{1}{\sqrt{2}}\begin{ssmatrix} 1 & 1 \\ i & -i \end{ssmatrix}$.
  Then
form   $(U;A,B) = \frac{1}{2}\begin{ssmatrix} 1 & 1 & -1& -1 \\ 1 & -1 & -i & i \\ 1 & 1 & 1 & 1 \\ 1 & -1 & i & -i \end{ssmatrix}$.
Thus $2.(U;A,B)$ is a Hadamard $H(4,4)$ matrix.
  
  Let $F= u_1u_1^*, F_2 = u_2u_u^*, F_3 = u_3u_3^*, F_4 = u_4u_4^*$ where $\{u_1, u_2, u_3,u_4\}$ are the columns of $ (U;A,B)$.

  Then
  $\left(\begin{smallmatrix} F_1\al_1 & F_2\al_2 & F_3\al_3 & F_4\al_4
    \\ F_2\al_5 & F_3\al_6 & F_4\al_7 & F_1\al_8 \\ F_3\al_9 &
    F_4\al_{10} & F_1\al_{11} & F_2\al_{12}  \\ F_4\al_{13}1 &
    F_1\al_{14} & F_2\al_{15} & F_3\al_{16} \end{smallmatrix}\right)$,
  for variables $\al_i$, is a paraunitary  matrix; this  is a unitary matrix when the variables are given values of modulus $1$. Also $F_1\al_1+F_2\al_2 + F_3\al_3 + F_3\al_4$ is a paraunitary matrix when the variables are given values  of modulus $1$.

  The process may be continued and infinite sequences obtained.} 
  \end{example} 
\subsubsection{Infinite sequences}\label{inf}

Let $\mathcal{P}$ be a property which is preserved by a matrix tangle product. 
Infinite series of entangled matrices with property $\mathcal{P}$ may
be obtained from constructions already given.  Here we give some more
general methods. Example \ref{three} above  gives the flavour. The
methods lead easily to strong encryption techniques including public
key systems. Error correcting codes may also be developed and both
encryption and error-correcting may be included in the one
system. 

Construct  infinite sequences of entangled matrices with property
$\mathcal{P}$ using initially two matrices with property
$\mathcal{P}$ as follows. 
Let  $A_1,A_2$ be  $2\ti 2$ matrices with a property $\mathcal{P}$ which
is preserved by matrix  tangle product. Form the $4\ti 4$ (different) entangled matrices $(A_1;A_1,A_2)= A_{11}, (A_2;A_1,A_2)=A_{12}, (A_1;A_2,A_1)=A_{13}, (A_2;A_2,A_1)=A_{14}$ which then have property $\mathcal{P}$. Each of the 12 pairs $\{A_{1i},A_{1j}| i\neq j\}$ may be tangles with shuffler  $A_1$ or $A_2$ giving  $24$ new entangled matrix tangle products of size $8\ti 8$ with property $\mathcal{P}$. Choose 2 different  elements of these 24 and form tangle products with either $A_1$ or $A_2$ to get $16 \ti 16$. 
This can be continued indefinitely. At each stage, matrices with property $\mathcal{P}$ are obtained.

\begin{example} Infinite series with real entries may be
  obtained. Suppose the initial
matrices are  real orthogonal 
as for example $A_1=\frac{1}{\sqrt{2}}\begin{smatrix}1 & 1 \\ 1& -1\end{smatrix}, A_2 = \frac{1}{\sqrt{2}}\begin{smatrix}-1&1 \\ 1&1\end{smatrix}$ or more generally of the form $\begin{smatrix} \cos \theta & \sin \theta \\ -\sin \theta & \cos \theta \end{smatrix}$ for differing $\theta$. 
\end{example}
\begin{construction}
Let  $S=\{A_1, A_2, \ldots, A_k\}$ be a set of  size $t\ti t$ matrices
with property $\mathcal{P}$ and $U$ an  $n\ti n$ matrix with property
$\mathcal{P}$. Construct $(U;A_{i_1}, A_{i_2},\ldots,
A_{i_n})$ or $(A_{i_1}, A_{i_2},\ldots, A_{i_n};U)$ with $i_j\in
\{1,2,\ldots, k\}$. For example $\mathcal{P}$ could be the property of
being unitary and $U$ could be the $n\ti n$ unitary Fourier matrix. To
be non-separable it is necessary that   the $i_j$ not all be
equal. This constructs $nt\ti nt$ matrices with property
$\mathcal{P}$; the $A_i$ and $U$ can vary. Infinite series are obtained by
varying $n$.  Infinite series may also be obtained by applying the
construction again using the
matrices constructed which have property $\mathcal{P}$. 
Many  such different infinite sequences may be constructed. 
\end{construction}

\section{Unitary space time}\label{spacetime}
In section \ref{prod} construction methods were laid out for various
types of matrices  and applications to the design of unitary,
paraunitary and special types of these matrices were given. Here we
give applications to the design of {\em constellations of matrices}. 
The design problem for unitary space time
constellations is set out as follows in \cite{mult} and \cite{orig}: ``Let $M$ be
the number of transmitter antennas and $R$ 
the desired transmission rate. Construct a set $\mathcal{V}$ of $L =
2^{RM}$ unitary $M\times M$ matrices such that for any two distinct elements $A,B$
in $\mathcal{V}$, the quantity $|\det(A-B)|$ is as large as possible. Any
set $\mathcal{V}$ such that $|\det(A-B)|> 0$ for all distinct $A,B \in
\mathcal{V}$ is said to have {\em full diversity}.''

The number of transmitter antennas is the size $M$ of the matrices. 
The set $\mathcal{V}$ is known as a {\em constellation} and  the {\em quality} of the constellation 
is measured by

 $$\zeta_{\mathcal{V}} = \frac{1}{2} \min_{V_l, V_m \in \mathcal{V}, V_l \neq
    V_m} |\det (V_l-V_m)|^{\frac{1}{M}}$$

Methods for constructing constellations while determining their quality using orthogonal symmetric idempotent sets was initiated in \cite{ted2}. These  can now be expanded and further constellations obtained  using the constructions in Section \ref{prod}.  

The survey article \cite{ams} proposes division algebras for  this
area and, although different, some comparisons can be made with the
constructions here.

Let $\{A_1, A_2, \ldots , A_k\}$ be a constellation of $m\ti m$ matrices with quality $\zeta$ and let $U$ be a unitary matrix. Then
\begin{enumerate}
\item $\{ (U;A_{i_1}, A_{i_2}, \ldots , A_{i_k}) | (i_1,i_2, \ldots, i_k) \, \, \text{is a derangement of} (1,2,\ldots, k) \}$ is a constellation of $mk\ti mk$ matrices
  of quality $\zeta$. A derangement is a permutation such that no element appears in its original position. 
\item Let $\{U_i | i= 1,2,\ldots, s \}$ be a constellation of quality $\zeta$ of $k\ti k$ matrices and $\{A_1,A_2,\ldots,A_k\}$ any $k$ unitary $t\ti t $ matrices. Then $\{(U_i;A_1,A_2, \ldots, A_k) | i=1,2,\ldots,s\}$ is a constellation of $kt\ti kt$ matrices with quality also $\zeta$.
  \end{enumerate} 

Unitary matrices and paraunitary matrices are constructed according to
Proposition \ref{idem1} using a COSI  set  $\{E_1, E_2, \ldots, E_k\}$ and forming
\\ $G= \left(\begin{smallmatrix} E_{i_{11}}\al_{11} & E_{i_{12}}\al_{12} & \ldots & E_{i_{1k}}\al_{1k} \\ E_{i_{21}}\al_{21} & E_{i_{22}}\al_{22} & \ldots & E_{i_{2k}}\al_{2k} \\ \vdots & \vdots & \vdots & \vdots \\ E_{i_{k1}}\al_{k1} & E_{i_{k2}}\al_{k2} & \ldots & E_{i_{kk}}\al_{kk} \end{smallmatrix}\right)$  where $\{E_1,E_2, \ldots, E_k\}$ appear once in each row and column. 

Let $G= \left(\begin{smallmatrix} E_1 \al_1 & E_2 \al_2 \\ E_2\al_1 &
  E_1\al_2 \end{smallmatrix}\right)$ where $\{E_1,E_2\}$ is a COSI set
of $2 \ti 2$ matrices  and the $\al_i$ are elements in $\C$. Then $\det G= \al_1^2\al_2^2$. 
Let now $\al_i$ be $n^{th}$ roots of unity and then $ \left\{\left(\begin{smallmatrix} E_1 \al_1 & E_2 \al_2 \\ E_2\al_2 & E_1\al_2 \end{smallmatrix}\right)\right\}$ is a constellation which has full diversity when an $n^{th}$ root of $1$ appears just once in each block column.  Let $A= \left(\begin{smallmatrix} E_1 \al_1 & E_2 \al_2 \\ E_2\al_1 & E_1\al_2 \end{smallmatrix}\right), B= \left(\begin{smallmatrix} E_1 \be_1 & E_2 \be_2 \\ E_2\be_1 & E_1\be_2 \end{smallmatrix}\right) $. Then $|\det (A-B)| = | (\al_1-\be_1)^2(\al_2-\be_2)^2| = | (\al_1-\be_1)|^2|(\al_2-\be_2)|^2$. 

The following is well-known and is easily verified.  
\begin{lemma}\label{cos} Let $z=\cos \theta + i\sin \theta$. Then $|1-z| = 2|\sin \frac{\theta}{2}|$
\end{lemma}
\begin{corollary}\label{cos1} Let $\al = \om^i, \be=\om^j$ with $i\neq j$ and $\om= e^{\frac{2i\pi}{n}}$ is a primitive $n^{th}$ root of unity. Then $|\al-\be| = 2|\sin \theta|$ where $\theta = \frac{\pi(j-i)}{n}$.
\end{corollary}

Now from Corollary \ref{cos1}, $|\det(A-B)| \geq 2^4|\sin \theta |^4$ where $\theta = \frac{\pi}{n}$. Thus the quality of the constellation is $\frac{1}{2}(2^4(|\sin \theta)|^4)^{\frac{1}{4}} 
= |\sin \theta|$.

The number that can be in each constellation when $n^{th}$ roots of unity are used is $n$. 
For $n=4$, $\theta = \frac{\pi}{4}$ and the quality $\approx 0.70710..$; the rate is $\frac{1}{2}$. For $n=8$, $\theta = \frac{\pi}{8}$ and the quality is $\approx 0.38268...$; the rate is $\frac{3}{4}$. For $n=16$, $\theta = \frac{\pi}{16}$ and the quality is $\approx 0.19509...$; the rate is $1$.

\quad


Higher order constellations may also be designed, and quality
determined explicitly, as follows.

Let   $G= \left(\begin{smallmatrix} E_1 \al_1 & E_2 \al_2 &\ldots & E_n \al_n \\ E_n \al_1 & E_{1}\al_{2} &\ldots & E_{n-1}\al_n \\ \vdots & \vdots & \vdots & \vdots &
\\ E_n\al_1 & E_{n-1}\al_2 & \ldots
&E_1\al_n \end{smallmatrix}\right)$ where $\{E_1,E_2,\ldots , E_n\}$
is a COSI set and the $\al_i$ are elements in $\C$. Then it may be
shown that $|\det(G)|= |\al_1\al_2 \ldots \al_n|^n$, where $n$ is the size
of the matrix $E_j$. The set of all $\left\{\left(\begin{smallmatrix} E_1 \al_1 & E_2 \al_2 &\ldots & E_n \al_n \\ E_n \al_1 & E_{1}\al_{2} &\ldots & E_{n-1}\al_n \\ \vdots & \vdots & \vdots & \vdots &
\\ E_n\al_1 & E_{n-1}\al_2 & \ldots
&E_1\al_n \end{smallmatrix}\right)\right\}$ with the $|\al_i|= 1$ is
then a constellation  of unitary matrices.  In particular let the
$\al_j$ be $n^{th}$ of unity such that no $\al_j$ appears in more than
one block column. Then the quality of this constellation is $|\sin
\theta|$ where $\theta = \frac{\pi}{n}$. Many such different  constellations 
with good quality may be formed. 




\end{document}